\documentclass[11pt, reqno]{amsart}
\usepackage{amsmath,amssymb,amsfonts,amsthm,graphicx,mathrsfs,mathtools}
\usepackage[margin=1.5in]{geometry}
\usepackage[dvipsnames]{xcolor}
\usepackage{graphicx,tipa}
\usepackage{tikz}
\usepackage{graphicx,amsmath,calc}
\usepackage{subcaption}
\makeatletter
\@namedef{subjclassname@2020}{%
\textup{2020} Mathematics Subject Classification}
\makeatother

\title[Frankel's property in the Schwarzschild manifolds]{Frankel's property for free boundary minimal hypersurfaces in the Riemannian Schwarzschild manifolds}

\author[J. Lee]{Jaehoon Lee}
\address[]{Jaehoon Lee, School of Mathematics, Korea Institute for Advanced Study, 85 Hoegiro, Dongdaemun-gu, Seoul 02455, Republic of Korea}
\email{jaehoonlee@kias.re.kr}

\author[E. Yeon]{Eungbeom Yeon}
\address[]{Eungbeom Yeon, Department of Mathematical Sciences, Pusan National University, Busan 46241, Republic of Korea}
\email{ebeom.yeon@pusan.ac.kr}

\begin{document}

\newtheorem{theorem}{theorem}[section]
\newtheorem{thm}[theorem]{Theorem}
\newtheorem{lemma}[theorem]{Lemma}
\newtheorem{ex}[theorem]{Example}
\newtheorem{cor}[theorem]{Corollary}
\newtheorem{prop}[theorem]{Proposition}
\newtheorem{rmk}[theorem]{Remark}
\newtheorem{Question}[theorem]{Question}
\newtheorem{conj}[theorem]{Conjecture}
\newtheorem*{mainthm1}{Theorem 1}
\newtheorem*{mainthm2}{Theorem 2}

\renewcommand{\theequation}{\thesection.\arabic{equation}}
\newcommand{\RNum}[1]{\uppercase\expandafter{\romannumeral #1\relax}}
\newcommand{\R}{\mathbb{R}}
\newcommand{\C}{\mathbb{C}}
\newcommand{\grad}{\nabla}
\newcommand{\laplacian}{\Delta}
\newcommand{\tgamma}{\tilde{\gamma}}
\newcommand{\ttau}{\tilde{\tau}}
\newcommand{\pp}{\Phi}
\newcommand{\tkappa}{\tilde{\kappa}}
\renewcommand{\d}{\textup{d}}
\newlength{\mywidth}
\newcommand\bigfrown[2][\textstyle]{\ensuremath{%
  \array[b]{c}\text{\resizebox{\mywidth}{.7ex}{$#1\frown$}}\\[-1.3ex]#1#2\endarray}}
\newcommand{\arc}[1]{{%
  \setbox9=\hbox{#1}%
  \ooalign{\resizebox{\wd9}{\height}{\texttoptiebar{\phantom{A}}}\cr#1}}}


\begin{abstract}
We study the behavior of minimal hypersurfaces in the Schwarzschild $n$-manifolds that intersect the horizon orthogonally along the boundary. We show that a free boundary minimal hypersurface and a totally geodesic hyperplane must intersect when the distance between them is achieved in a bounded region. We also discuss when the Schwarzschild metric is perturbed in a way that its scalar curvature is no longer positive.
\end{abstract}

\maketitle

\section{Introduction}\label{intro}
\setcounter{equation}{0}
In the theory of general relativity, the Schwarzschild solution to an Einstein field equation is an important model in which one can describe local geometry of space-time in the solar system. It represents an empty space-time outside the spherical mass, which also has spherical symmetry and asymptotically flat behavior (\cite{HE}). The Schwarzschild solution can be described as follows:
\begin{align*}
M^n=\left\{ x \in \mathbb{R}^n : |x| \ge \left( \frac{m}{2}\right)^{\frac{1}{n-2}} \right\},\  g^n_{Sch}=\left( 1+\frac{m}{2|x|^{n-2}}\right)^{\frac{4}{n-2}}g_{\mathbb{R}^n}.
\end{align*}
Here, $|x| = g_{\mathbb{R}^n}\left(x,x\right)^{\frac{1}{2}}$  and $m$ is defined as an ADM mass of the manifold (see \cite{C} for the definition). The spherical mass body is called a horizon, denoted as $\partial 
M^n$ throughout the paper, and it is a totally geodesic hypersurface in $M^n$. Additional examples of embedded minimal hypersurfaces are totally geodesic hyperplanes orthogonal to $\partial M^n$. It should be noted that the Schwarzchild solution is the only radially symmetric solution to the Einstein equation. As the term on the metric is the Green function of Laplacian, we see that $M^n$ is scalar flat. 

One of the most important aspects of the Riemannian Schwarzschild manifold $M^n$ is to study minimal submanifolds in $M^n$. Schoen and Yau in \cite{SY} proved that asymptotically planar stable minimal surfaces cannot exist in asymptotically Schwarzschildean manifolds. In \cite{C}, it was generalized to non-existence results of complete stable properly embedded minimal surfaces. As mentioned in \cite{C}, properness in the theorem can be omitted if there exists a horizon boundary in the manifold (\cite{CCE}). In fact, positivity of ADM mass plays an important role in obstructing the existence of stable minimal surfaces. In \cite{CM}, again non-existence results of minimal surfaces in asymptotically Schwarzschildean 3-manifolds were given. They showed that there is no such a minimal surface perturbing the Euclidean catenoid in asymptotically Schwarzschildean 3-manifolds. It is worthwhile noting that the obstruction does not hold for higher dimensional manifolds in \cite{C}. 

In \cite{CK}, obstruction of the horizon became more clear as they showed that there exist complete properly embedded minimal planes passing through an arbitrary given point in asymptotically flat 3-manifolds assuming the absence of closed minimal surfaces. One can also consider uniqueness and existence problems on complete minimal surfaces with a number of boundaries supported on some constraint surface. It was noted in \cite{C} that the above non-existence results also apply to the stable free boundary minimal surfaces with weakly mean convex boundary supported on the horizon in asymptotically Schwarzschildean 3-manifolds. 

Indeed, a question was raised in \cite{CK} that asks if the horizon and totally geodesic planes are the only embedded minimal surfaces in the Riemannian Schwarzschild manifold $M^n$. Brendle \cite{B} showed the analogous version of the problem for closed constant mean curvature hypersurfaces in manifolds with certain warped product metrics, including the Schwarzschild metric. The above question is interesting in the following manner. 
In this paper, we are mainly interested in how the metric affects the behavior of free boundary minimal hypersurfaces in the Schwarzschild manifolds. First, we get the following theorem. 
\begin{mainthm1} \label{mainthm1} \text{\normalfont (Theorem \ref{thm1})}
Let $\Sigma^{n-1}$ be a smooth connected two-sided properly immersed minimal hypersurface in the Riemannian Schwarzschild $n$-manifold $(M^n, g^n_{Sch})$ such that $\partial \Sigma^{n-1} \subset \partial M^n$ and $\Sigma^{n-1}$ meets $\partial M^n$ orthogonally along $\partial \Sigma^{n-1}$. Suppose that the distance between $\Sigma^{n-1}$ and a totally geodesic plane $P$ is achieved in a bounded region in $M^n$. Then then $\Sigma^{n-1}$ and $P$ must intersect. 
\end{mainthm1}

As an immediate corollary to the above theorem, we get the following half-space theorem.
\begin{mainthm2} \label{mainthm2} \text{\normalfont (Corollary \ref{cor1})}
Let $\Sigma^{n-1}$ and $(M^n, g^n_{Sch})$ be as above. Suppose that $\Sigma^{n-1}$ is on one side of a totally geodesic hyperplane $P$ and the distance between $\Sigma^{n-1}$ and $P$ is achieved in a bounded region in $M^n$. Then $\Sigma^{n-1}=P$.
\end{mainthm2}

Note that Theorem $2$ could be seen as a half-space theorem in the Schwarzschild manifold. Contrary to cases of minimal hypersurfaces in $\mathbb{R}^n$ where we can get half-space theorems applying maximum principle at infinity using catenoidal piece (\cite{HM}), the Schwarzschild manifolds are harder to apply maximum principle since there are few minimal examples known in addition to the fact that the minimality is not invariant under translations and scaling in $M^n$. The main idea of proving Theorem $1$ was to apply Frankel's argument on geodesics in $M^n$. As in the literature (for example, see \cite{F,L}), Frankel's argument was normally adapted under the variation of geodesics when the ambient manifold has positive Ricci curvature. It is very unlikely to expect the argument could be applied to prove the result.

We should mention that following arguments in the proof of the slab theorem (\cite{CCE}) one can get the same result in Theorem \ref{mainthm1} in the absence of the distance condition when $n=3$. However, we can perturb the Schwarzschild metric to show that Theorem \ref{mainthm1} holds in metrics whose scalar curvature can be both negative and positive. This represents that our method points a new direction compared to \cite{CCE} despite the condition on the distance at infinity. See Section \ref{sec4} for more details.
 
The paper is organized as follows. In Section \ref{pre} we give a detailed description of the Riemannian Schwarzschild manifolds, including properties of geodesics. New observations on Ricci curvature along geodesics are given so that one can see how the Schwarzschild metric affects the geometry of minimal hypersurfaces in a global way. With the help of these observations, we prove our main result (Theorem \ref{thm1}) which directly allows us to prove the half-space theorem (Corollary \ref{cor1}) in Section \ref{sec3}. In the last section, we give a way to perturb the Schwarzschild metric to expand our results to more general asymptotically flat metrics.

\section*{Acknowledgements}
We would like to thank Jaigyoung Choe for his interest and valuable comments. We also wish to express our gratitude to Hojoo Lee for raising a question on minimal cones in Remark \ref{conecone} and to a reviewer for pointing out \cite[Theorem 1.7]{CCE} in Remark \ref{slabslab} during the peer-review process. J. Lee was supported by a KIAS Individual Grant MG086401 at Korea Institute for Advanced Study. E. Yeon was supported by National Research Foundation of Korea NRF-2022R1C1C2013384 and in part by NRF-2021R1A4A1032418.


\section{Geometry of the Riemannian Schwarzschild manifold}\label{pre}
\subsection{Doubled Schwarzschild manifolds}
Let $\widehat{M^n}=\mathbb{R}^n\backslash\{(0,0,\cdots,0)\}$ ($n\geq3$) be a Riemannian manifold with a metric $\widehat{g^n}:=\left(1+\frac{m}{2|x|^{n-2}}\right)^{\frac{4}{n-2}}g_{\mathbb{R}^n}$, where $g_{\mathbb{R}^n}$ is the Euclidean metric and $|\cdot|$ is the Euclidean norm. Consider a map $I: \widehat{M^n}\to\widehat{M^n}$ given by
\begin{align*}
I(x)=\left(\frac{m}{2}\right)^{\frac{2}{n-2}}\cdot\frac{x}{|x|^2}, \ \forall x\in\mathbb{R}^n\backslash\{(0,0,\cdots, 0)\},
\end{align*}
which is an inversion with respect to the Euclidean sphere $\left\{x\in\mathbb{R} ^n:|x|=\left(\frac{m}{2}\right)^{\frac{1}{n-2}}\right\}$. Then it gives rise to an isometry of $\widehat{M^n}$. This follows from the fact that $I$ is a conformal map in the Euclidean space with the conformal factor $\left(\frac{m}{2}\right)^{\frac{4}{n-2}}\frac{1}{|x|^4}$ and $I$ satisfies 
\begin{align*}
\left(1+\frac{m}{2\left|I(x)\right|^{n-2}}\right)^{\frac{4}{n-2}}\cdot\left(\frac{m}{2}\right)^{\frac{4}{n-2}}\frac{1}{|x|^4}=\left(1+\frac{m}{2|x|^{n-2}}\right)^{\frac{4}{n-2}}.
\end{align*}
Since the Riemannian Schwarzschild $n$-manifold $(M^n, g^n_{Sch})$ is defined by
\begin{align*}
M^n=\left\{x\in\mathbb{R}^n:|x|\geq \left(\tfrac{m}{2}\right)^{\frac{1}{n-2}}\right\},\  g^n_{Sch}=\left(1+\tfrac{m}{2|x|^{n-2}}\right)^{\frac{4}{n-2}}g_{\mathbb{R}^n},
\end{align*}
we may write $\widehat{M^n}=M^n\cup I(M^n)$. Therefore $\widehat{M^n}$ is often called the \emph{doubled Schwarzschild $n$-manifold}. There are totally geodesic hypersurfaces in $M^n$ as sets of fixed points under the isometries:
\begin{itemize}
\item[-] The horizon $\partial M^n=\left\{x\in\mathbb{R}^n:|x|=\left(\frac{m}{2}\right)^{\frac{1}{n-2}}\right\}$. 
\item[-] Hyperplanes passing through the origin.
\end{itemize}
One can prove that these are the only totally geodesic hypersurfaces in $\widehat{M^n}$. Indeed, being a totally geodesic in $\widehat{M^n}$ implies being a totally umbilic in the Euclidean space, thus becoming part of an $(n-1)$-sphere or a hyperplane. Among spheres and hyperplanes, it can be seen that only the horizon and hyperplanes passing through the origin are totally geodesic in the doubled Schwarzschild manifold.
\subsection{Geodesics}
Let us introduce a function $\phi_{n,m}$ and express the conformal factor as $e^{2\phi_{n,m}(x)}=\left(1+\frac{m}{2|x|^{n-2}}\right)^{\frac{4}{n-2}}$ for notational convenience. We describe geodesics in $(M^n, g^n_{Sch})$ that are perpendicular to totally geodesic hyperplanes. Note that each of them is contained in a $2$-plane determined by the initial position vector and the normal direction of the totally geodesic hyperplane. Therefore it is sufficient to consider the following forms: 
\begin{align}\label{gamma}
\begin{cases}
&\gamma(s)=r(s)\left(\cos\theta(s), \sin\theta(s), 0,\cdots,0\right),\\
&\theta(0)=0,\ \dot{\theta}(0)>0,\ r(0)=r_0,\ \dot{r}(0)=0.
\end{cases} 
\end{align}
Here, $s$ is an arclength and $r_0\geq\left(\frac{m}{2}\right)^{\frac{1}{n-2}}$. Since
\begin{align}\label{dgamma}
\dot{\gamma}=\dot{r}\left(\cos\theta,\sin\theta,0,\cdots,0\right)+r\dot{\theta}\left(-\sin\theta,\cos\theta,0,\cdots,0\right),
\end{align}
we have
\begin{align}\label{arclength}
1=g^n_{Sch}(\dot{\gamma},\dot{\gamma})=e^{2\phi_{n,m}(\gamma)}g_{\mathbb{R}^n}(\dot{\gamma},\dot{\gamma})=e^{2\phi_{n,m}(\gamma)}\left(\dot{r}^2+r^2\dot{\theta}^2\right)
\end{align}
and 
\begin{align}\label{init}
r_0\dot{\theta}(0)=e^{-\phi_{n,m}(\gamma(0))}=\left(1+\frac{m}{2r_0^{n-2}}\right)^{-\frac{2}{n-2}}.
\end{align}
Recall that the Riemannian connection $\nabla$ of $M^n$ is given by
\begin{align}\label{conn}
\nabla_XY=\overline{\nabla}_XY+d\phi_{n,m}(X)Y+d\phi_{n,m}(Y)X-g_{\mathbb{R}^n}(X,Y)\overline{\nabla}\phi_{n,m},
\end{align}
where $\overline{\nabla}$ is the Riemannian connection of the Euclidean space and $X,Y$ are tangent vector fields. Now the geodesic equation for $\gamma$ becomes
\begin{align}\label{geo}
\nabla_{\dot{\gamma}}\dot{\gamma}&=\overline{\nabla}_{\dot{\gamma}}\dot{\gamma}+2d\phi_{n,m}(\dot{\gamma})\dot{\gamma}-g_{\mathbb{R}^n}(\dot{\gamma}, \dot{\gamma})\overline{\nabla}\phi_{n,m}\nonumber\\
&=\ddot{\gamma}+2g_{\mathbb{R}^n}(\overline{\nabla}\phi_{n,m}, \dot{\gamma})\dot{\gamma}-e^{-2\phi_{n,m}(\gamma)}\overline{\nabla}\phi_{n,m}=0,
\end{align}
where we used (\ref{arclength}) in the second equality. (\ref{geo}) together with 
\begin{align*}
\overline{\nabla}\phi_{n,m}=-\frac{m}{r^n\left(1+\frac{m}{2r^{n-2}}\right)}\gamma,
\end{align*}
we obtain
\begin{align}\label{ggeo}
\ddot{\gamma}-\frac{2m\dot{r}}{r^{n-1}\left(1+\frac{m}{2r^{n-2}}\right)}\dot{\gamma}+\frac{m}{r^n\left(1+\frac{m}{2r^{n-2}}\right)^{\frac{n+2}{n-2}}}\gamma=0.
\end{align}
Using (\ref{gamma}), (\ref{dgamma}) and 
\begin{align*}
\ddot{\gamma}=\left(\ddot{r}-r\dot{\theta}^2\right)\left(\cos\theta,\sin\theta,0,\cdots,0\right)+\left(2\dot{r}\dot{\theta}+r\ddot{\theta}\right)\left(-\sin\theta,\cos\theta,0,\cdots,0\right),
\end{align*}
(\ref{ggeo}) is expressed as
\begin{align*}
&\left(\ddot{r}-r\dot{\theta}^2-\frac{2m\dot{r}^2}{r^{n-1}\left(1+\frac{m}{2r^{n-2}}\right)}+\frac{m}{r^{n-1}\left(1+\frac{m}{2r^{n-2}}\right)^{\frac{n+2}{n-2}}}\right)\left(\cos\theta,\sin\theta,0,\cdots,0\right)\\
&+\left(2\dot{r}\dot{\theta}+r\ddot{\theta}-\frac{2m\dot{r}\dot{\theta}}{r^{n-2}\left(1+\frac{m}{2r^{n-2}}\right)}\right)\left(-\sin\theta,\cos\theta,0,\cdots,0\right)=0.
\end{align*}
This implies
\begin{align}\label{first}
\ddot{r}-r\dot{\theta}^2-\frac{2m\dot{r}^2}{r^{n-1}\left(1+\frac{m}{2r^{n-2}}\right)}+\frac{m}{r^{n-1}\left(1+\frac{m}{2r^{n-2}}\right)^{\frac{n+2}{n-2}}}=0
\end{align}
and
\begin{align}\label{second}
2\dot{r}\dot{\theta}+r\ddot{\theta}-\frac{2m\dot{r}\dot{\theta}}{r^{n-2}\left(1+\frac{m}{2r^{n-2}}\right)}=0.
\end{align}
It follows from (\ref{second}) that
\begin{align*}
&\frac{d}{ds}\left(r^2\left(1+\frac{m}{2r^{n-2}}\right)^{\frac{4}{n-2}}\dot{\theta}\right)\\
&=r\left(1+\frac{m}{2r^{n-2}}\right)^{\frac{4}{n-2}}\left(2\dot{r}\dot{\theta}+r\ddot{\theta}-\frac{2m\dot{r}\dot{\theta}}{r^{n-2}\left(1+\frac{m}{2r^{n-2}}\right)}\right)\\
&=0
\end{align*}
and thus
\begin{align}\label{dtheta}
r^2\left(1+\frac{m}{2r^{n-2}}\right)^{\frac{4}{n-2}}\dot{\theta}=r_0^2\left(1+\frac{m}{2r_0^{n-2}}\right)^{\frac{4}{n-2}}\dot{\theta}(0)=:C_0.
\end{align}
Note that $C_0=r_0\left(1+\frac{m}{2r_0^{n-2}}\right)^{\frac{2}{n-2}}$ by (\ref{init}). 

It is well known that geodesics are complete in $M^n$ because the metric on $M^n$ is conformal to the Euclidean one. However, we prove the following lemma to see a more detailed aspect of such geodesics:
\begin{lemma}\label{rinfty}
If $r_0>\left(\frac{m}{2}\right)^{\frac{1}{n-2}}$, then $r(s)$ is a strictly increasing function of $s$ and $\displaystyle\lim_{s\to+\infty}r(s)=+\infty$.
\end{lemma}
\begin{proof}
By (\ref{gamma}), (\ref{init}) and (\ref{first}),
\begin{align*}
\ddot{r}(0)=\frac{r_0^{n-2}-\frac{m}{2}}{r_0^{n-1}\left(1+\frac{m}{2r_0^{n-2}}\right)^{\frac{n+2}{n-2}}}>0.
\end{align*}
Hence $\dot{r}(s)>0$ for $s$ near $0$. Using (\ref{arclength}) and (\ref{dtheta}), we derive
\begin{align*}
\dot{r}^2=\frac{r^2\left(1+\frac{m}{2r^{n-2}}\right)^{\frac{4}{n-2}}-C_0^2}{r^2\left(1+\frac{m}{2r^{n-2}}\right)^{\frac{8}{n-2}}}
\end{align*}
and therefore
\begin{align}\label{rdot}
\dot{r}=\frac{\sqrt{r^2\left(1+\frac{m}{2r^{n-2}}\right)^{\frac{4}{n-2}}-C_0^2}}{r\left(1+\frac{m}{2r^{n-2}}\right)^{\frac{4}{n-2}}}
\end{align}
for $s$ near $0$. Observing (\ref{rdot}) $r^2\left(1+\frac{m}{2r^{n-2}}\right)^{\frac{4}{n-2}}-C_0^2$ is a strictly increasing function of $r$ if $r>\left(\frac{m}{2}\right)^{\frac{1}{n-2}}$ and is zero at $s=0$, so it cannot achieve zero again as long as (\ref{rdot}) holds. Thus, we may conclude that (\ref{rdot}) holds for all $s\geq0$. Moreover, the right hand side of (\ref{rdot}) is also an increasing function of $r$, so it is clear that $r\to+\infty$ as $s$ increases.
\end{proof}

\subsection{Ricci curvature along geodesics}
We continue to deal with geodesics as in (\ref{gamma}) and show a new observation in Proposition \ref{ricciineq} that plays an essential role in this paper. 

Let $\gamma$ be a geodesic in $M^n$ as in (\ref{gamma}). To begin with, we compute the Ricci curvature of $M^n$ along $\gamma$. It is not difficult to derive from (\ref{conn}) that 
\begin{align*}
\text{Ric}_{M^n}(\dot{\gamma}, \dot{\gamma})=&\overline{\text{Ric}}(\dot{\gamma},\dot{\gamma})-(n-2)\overline{\text{Hess}}\phi_{n,m}(\dot{\gamma},\dot{\gamma})+(n-2)\left(d\phi_{n,m}(\dot{\gamma})\right)^2\\
&-(n-2)\left|\overline{\nabla}\phi_{n,m}\right|^2g_{\mathbb{R}^n}(\dot{\gamma},\dot{\gamma})-\overline{\Delta}\phi_{n,m} \cdot g_{\mathbb{R}^n}(\dot{\gamma},\dot{\gamma}),
\end{align*}
where all the over-lined terms correspond to the Euclidean metric. The following result is just a simple computation, so we only state the formula without proof:
\begin{lemma}\label{ricci}
\begin{align*}
\emph{Ric}_{M^n}(\dot{\gamma}, \dot{\gamma})=\frac{m(n-2)}{r^n\left(1+\frac{m}{2r^{n-2}}\right)^{\frac{2n}{n-2}}}\left[\frac{nC_0^2}{r^2\left(1+\frac{m}{2r^{n-2}}\right)^{\frac{4}{n-2}}}-(n-1)\right],
\end{align*}
where $C_0$ is a constant defined in (\ref{dtheta}).
\end{lemma}
Using (\ref{dtheta}), (\ref{rdot}) and the behavior of $r=r(s)$ proved in Lemma \ref{rinfty}, one can easily confirm that $\gamma$ eventually approaches asymptotically to Euclidean geodesic lines. However, the Ricci curvature has a negative sign in radial directions. Therefore the following inequality is not straightforward to predict and depends heavily on the global properties of the Schwarzschild metrics.
\begin{prop}\label{ricciineq}
For any $r_0\geq\left(\frac{m}{2}\right)^{\frac{1}{n-2}}$ and $d>0$,
\begin{align*}
\int_0^d-\emph{Ric}_{M^n}(\dot{\gamma},\dot{\gamma})\d s<0.
\end{align*}
\end{prop}
\begin{proof}
If $r_0=\left(\frac{m}{2}\right)^{\frac{1}{n-2}}$, the geodesic will continue to exist on the horizon. In this case, $\text{Ric}_{M^n}(\dot{\gamma}, \dot{\gamma})$ is a positive constant $\frac{n-2}{2^{\frac{n}{n-2}}m^{\frac{2}{n-2}}}$, so the inequality immediately follows. 

Suppose that $r_0>\left(\frac{m}{2}\right)^{\frac{1}{n-2}}$. As it can be seen from Lemma \ref{ricci}, the Ricci curvature is positive at the beginning and changes to be negative when $r$ is sufficiently large. Because we proved that $r$ is an increasing function of $s$ in Lemma \ref{rinfty}, it is sufficient to prove the inequality for the case when $d\to\infty$. Since
\begin{align*}
\int_0^{\infty}-\emph{Ric}_{M^n}(\dot{\gamma},\dot{\gamma})\d s=\int_0^{\infty}\frac{m(n-2)}{r^n\left(1+\frac{m}{2r^{n-2}}\right)^{\frac{2n}{n-2}}}\left[(n-1)-\frac{nC_0^2}{r^2\left(1+\frac{m}{2r^{n-2}}\right)^{\frac{4}{n-2}}}\right]\d s
\end{align*}
we can finish the proof if we show the following inequality:
\begin{align}\label{ineq}
\int_0^{\infty}\frac{1}{r^n\left(1+\frac{m}{2r^{n-2}}\right)^{\frac{2n}{n-2}}}\left[(n-1)-\frac{nC_0^2}{r^2\left(1+\frac{m}{2r^{n-2}}\right)^{\frac{4}{n-2}}}\right]\d s<0.
\end{align} 
By Lemma \ref{rinfty}, we may substitute $s$ by $r$. We have
\begin{align*}
\d s=\frac{1}{\dot{r}}\d r=\frac{r\left(1+\frac{m}{2r^{n-2}}\right)^{\frac{4}{n-2}}}{\sqrt{r^2\left(1+\frac{m}{2r^{n-2}}\right)^{\frac{4}{n-2}}-C_0^2}}\d r
\end{align*}
from (\ref{rdot}), and the integral in (\ref{ineq}) becomes
\begin{align*}
\int_{r_0}^{\infty}\frac{1}{r^{n-1}\left(1+\frac{m}{2r^{n-2}}\right)^2}\left[(n-1)-\frac{nC_0^2}{r^2\left(1+\frac{m}{2r^{n-2}}\right)^{\frac{4}{n-2}}}\right]\frac{1}{\sqrt{r^2\left(1+\frac{m}{2r^{n-2}}\right)^{\frac{4}{n-2}}-C_0^2}}\d r.
\end{align*}
We again substitute $r\left(1+\frac{m}{2r^{n-2}}\right)^{\frac{2}{n-2}}$ by $u$. Then 
\begin{align*}
r^{n-2}=\frac{1}{2}\left(u^{n-2}-m+\sqrt{u^{n-2}(u^{n-2}-2m)}\right)
\end{align*}
and
\begin{align*}
\d r=\frac{u^{n-3}}{2r^{n-2}}\left(1+\frac{u^{n-2}-m}{\sqrt{u^{n-2}(u^{n-2}-2m)}}\right)\d u=\frac{ru^{n-3}}{\sqrt{u^{n-2}(u^{n-2}-2m)}}\d u.
\end{align*}
Therefore we can express the integral as
\begin{align*}
\int_{C_0}^{\infty}\frac{1}{u}\left[(n-1)-\frac{nC_0^2}{u^2}\right]\frac{1}{\sqrt{u^2-C_0^2}}\frac{1}{\sqrt{u^{n-2}(u^{n-2}-2m)}}\d u.
\end{align*}
Finally, if we substitute $u$ by $C_0t$, the integral is expressed as
\begin{align*}
\frac{1}{C_0^{n-1}}\int_1^{\infty}\frac{1}{t}\left[(n-1)-\frac{n}{t^2}\right]\frac{1}{\sqrt{t^2-1}}\frac{1}{\sqrt{t^{n-2}(t^{n-2}-\alpha)}}\d t,
\end{align*} 
where $\alpha:=\frac{2m}{C_0^{n-2}}\in(0, 1)$. Now proving (\ref{ineq}) is equivalent to proving
\begin{align*}
\int_1^{\infty}\frac{1}{t}\left[(n-1)-\frac{n}{t^2}\right]\frac{1}{\sqrt{(t^2-1)t^{n-2}(t^{n-2}-\alpha)}}\d t<0,\ \forall \alpha\in(0, 1).
\end{align*}
We show this inequality in Lemma \ref{lembel} below, and we have the result.
\end{proof}

\begin{lemma}\label{lembel}
Let $\alpha\in(0, 1)$. Then
\begin{align*}
\int_1^{\infty}\frac{1}{t}\left[(n-1)-\frac{n}{t^2}\right]\frac{1}{\sqrt{(t^2-1)t^{n-2}(t^{n-2}-\alpha)}}\d t<0.
\end{align*}
\end{lemma}
\begin{proof}
Substitute $t$ by $\frac{1}{v}$. Then
\begin{align*}
\int_1^{\infty}\frac{1}{t}\left[(n-1)-\frac{n}{t^2}\right]\frac{1}{\sqrt{(t^2-1)t^{n-2}(t^{n-2}-\alpha)}}\d t=\int_0^1\frac{v^{n-2}\left((n-1)-nv^2\right)}{\sqrt{(1-v^2)(1-\alpha v^{n-2})}}\d v,
\end{align*}
and by letting $v=\sin\psi$, we derive
\begin{align*}
\int_0^1\frac{v^{n-2}\left((n-1)-nv^2\right)}{\sqrt{(1-v^2)(1-\alpha v^{n-2})}}\d v=\int_0^{\frac{\pi}{2}}\frac{(n-1)\sin^{n-2}\psi-n\sin^n\psi}{\sqrt{1-\alpha\sin^{n-2}\psi}}\d\psi.
\end{align*}
Using the fact that the series
\begin{align*}
\frac{1}{\sqrt{1-x}}=1+\sum_{j=0}^{\infty}\frac{1}{2^{2j+1}}\binom{2j+1}{j}x^{j+1}
\end{align*}
converges uniformly in $|x|\leq\alpha$, we may write
\begin{align*}
\int_0^{\frac{\pi}{2}}\frac{(n-1)\sin^{n-2}\psi}{\sqrt{1-\alpha\sin^{n-2}\psi}}\d\psi&=(n-1)\int_0^{\frac{\pi}{2}}\sin^{n-2}\psi \d\psi\\&+\sum_{j=0}^{\infty}\frac{1}{2^{2j+1}}\binom{2j+1}{j}\alpha^{j+1}\left(\int_0^{\frac{\pi}{2}}(n-1)\sin^{(j+2)(n-2)}\psi \d\psi\right)
\end{align*}
and
\begin{align*}
\int_0^{\frac{\pi}{2}}\frac{n\sin^n\psi}{\sqrt{1-\alpha\sin^{n-2}\psi}}\d\psi&=n\int_0^{\frac{\pi}{2}}\sin^n\psi \d\psi\\
&+\sum_{j=0}^{\infty}\frac{1}{2^{2j+1}}\binom{2j+1}{j}\alpha^{j+1}\left(\int_0^{\frac{\pi}{2}}n\sin^{(j+2)(n-2)+2}\psi \d\psi\right).
\end{align*}
Since
\begin{align*}
n\int_0^{\frac{\pi}{2}}\sin^n\psi \d\psi=(n-1)\int_0^{\frac{\pi}{2}}\sin^{n-2}\psi \d\psi
\end{align*}
and
\begin{align*}
\int_0^{\frac{\pi}{2}}n\sin^{(j+2)(n-2)+2}\psi \d\psi&=n\frac{(j+2)(n-2)+1}{(j+2)(n-2)+2}\int_0^{\frac{\pi}{2}}\sin^{(j+2)(n-2)}\psi \d\psi\\
&>\int_0^{\frac{\pi}{2}}(n-1)\sin^{(j+2)(n-2)}\psi \d\psi,\ \forall j\geq0,
\end{align*}
we can conclude that
\begin{align*}
\int_0^{\frac{\pi}{2}}\frac{(n-1)\sin^{n-2}\psi-n\sin^n\psi}{\sqrt{1-\alpha\sin^{n-2}\psi}}\d\psi<0
\end{align*}
and the inequality is obtained.
\end{proof}

\begin{rmk}
\textup{In the literature, asymptotically flat metrics similar to the Schwarzschild ones also have been considered. They are defined as
\begin{align*}
g=\left(1+\frac{m}{2|x|^{\frac{n-2k}{k}}}\right)^{\frac{4k}{n-2k}}g_{\mathbb{R}^n},           
\end{align*}
where $n$ is a dimension and $k$ is a positive integer. We remark that a similar calculation also yields the same inequality as in Proposition \ref{ricciineq} provided that $n>2k\geq2$. One may refer to \cite{CTZ} for the physical background of these examples.}
\end{rmk}


\section{Free boundary minimal hypersurfaces in $M^n$}\label{sec3}
We prove that totally geodesic hyperplanes attract free boundary minimal hypersurfaces by applying Frankel's argument(\cite{F}):
\begin{thm} \label{thm1}
Let $\Sigma^{n-1}$ be a smooth connected two-sided properly immersed minimal hypersurface in the Riemannian Schwarzschild $n$-manifold $(M^n, g^n_{Sch})$ such that $\partial \Sigma^{n-1} \subset \partial M^n$ and $\Sigma^{n-1}$ meets $\partial M^n$ orthogonally along $\partial \Sigma^{n-1}$. Assume that the distance between $\Sigma^{n-1}$ and a totally geodesic hyperplane $P$ is achieved in a bounded region, i.e.,
\begin{align}\label{dist}
\inf_{R\to\infty}\emph{dist}_{M^n}(\Sigma^{n-1}\cap B_R, P\cap B_R)=\emph{dist}_{M^n}(\Sigma^{n-1}\cap B_{R_0}, P\cap B_{R_0})
\end{align}
for some $R_0\in \left[\left(\frac{m}{2}\right)^{\frac{1}{n-2}}, \infty\right)$. Here, $B_R$ denotes the $n$-dimensional Euclidean ball of radius $R$ centered at the origin. Then $\Sigma^{n-1}$ and $P$ must intersect.
\end{thm}
\begin{proof}
Assume the contrary, that is, the distance is positive:
\begin{align*}
\inf_{R\to\infty}\text{dist}_{M^n}(\Sigma^{n-1}\cap B_R, P\cap B_R)=\text{dist}_{M^n}(\Sigma^{n-1}\cap B_{R_0}, P\cap B_{R_0})>0.
\end{align*}
Recall that the inversion $I$ with respect to $\partial M^n$ is an isometry of the doubled Schwarzschild manifold $\widehat{M^n}$ (see Section \ref{pre}). Both $\Sigma^{n-1}$ and $P$ meet $\partial M^n$ orthogonally, so we obtain smooth minimal hypersurfaces $\widehat{\Sigma^{n-1}}:=\Sigma^{n-1}\cup I(\Sigma^{n-1})$ and $\widehat{P}:=P\cup I(P)$ in $\widehat{M^n}$. By the assumption, the distance between $\widehat{\Sigma^{n-1}}$ and $\widehat{P}$ is also achieved in a bounded region in $\widehat{M^n}$ and is positive. Therefore there exists a geodesic $\gamma=\gamma(s): [0, d]\to \widehat{M^n}$ realizing the distance, where $s$ is an arclength parametrization and $\gamma(0)\in\widehat{P}$. Note that $\gamma$ is orthogonal to both $\widehat{\Sigma^{n-1}}$ and $\widehat{P}$. Then, by the uniqueness of geodesics, we may assume that $\gamma$ lies in $M^n$. (This can also be followed by the fact that $I$ is an isometry and $\gamma$ is minimizing.)

Now we apply Frankel's argument(\cite{F}) to $\gamma$. Choose orthonormal basis of $T_{\gamma(0)}P$ and extend them to parallel normal vector fields $\{E_1, E_2,\cdots, E_{n-1}\}$ along $\gamma$. For each $j$, consider a variation of $\gamma$ generated by $E_j$. The second variation formula for the arclength reads
\begin{align*}
\delta^2L_{\gamma}(E_j, E_j)&=g^n_{Sch}(\nabla_{E_j}E_j(\gamma(d)), \dot{\gamma}(d))\\&-g^n_{Sch}(\nabla_{E_j}E_j(\gamma(0)), \dot{\gamma}(0))-\int_0^d\text{R}_{M^n}(\dot{\gamma},E_j,E_j,\dot{\gamma})\d s.
\end{align*}
Summing up for all $j$,
\begin{align*}
0&\leq\sum_{j=1}^{n-1}\delta^2L_{\gamma}(E_j, E_j)\\
&=g^n_{Sch}\left(\vec{H}_{\Sigma}(\gamma(d)),\dot{\gamma}(d)\right)-g^n_{Sch}\left(\vec{H}_{P}(\gamma(0)),\dot{\gamma}(0)\right)-\int_0^d\text{Ric}_{M^n}(\dot{\gamma},\dot{\gamma})\d s\\
&=-\int_0^d\text{Ric}_{M^n}(\dot{\gamma},\dot{\gamma})\d s,
\end{align*}
where $\vec{H}$ is the mean curvature vector. Note that the first inequality comes from the fact that $\gamma$ is minimizing, and we used minimality of $\Sigma^{n-1}$ and $P$ in the last equality. However, we proved in Proposition \ref{ricciineq} that
\begin{align*}
-\int_0^d\text{Ric}_{M^n}(\dot{\gamma},\dot{\gamma})\d s<0
\end{align*}
for any $d>0$. This is a contradiction. Thus $\Sigma^{n-1}$ and $P$ must intersect.
\end{proof}
\begin{rmk}\label{conecone}
\normalfont We see that dimension of $P$ can be less than $n-1$ since it is totally geodesic. On the other hand, cones over minimal submanifolds in $\mathbb{S}^{n-1}$ also give us nontrivial minimal examples in $M^n$. In fact, we can set $P$ as codimension $1$ minimal cones to have the same result since tangent spaces of minimal cones are totally geodesic hyperplanes and minimizing geodesics are identical to the ones shown in Section \ref{pre}.
\end{rmk}
\begin{rmk}\label{slabslab}
\normalfont We should mention that \cite[Theorem 1.7]{CCE} provides us with more general perspectives regarding the case $n=3$ in the above theorem. Indeed, \cite[Theorem 1.7]{CCE} asserts that any slab bounded by complete minimal surfaces in an asymptotically flat $3$-manifold with non-negative scalar curvature and horizon boundary must be a Euclidean slab. A slight modification of their proof can be adapted to our case when $n=3$. One of the main ideas in their proof was to exploit non-existence of complete stable minimal surfaces. However, there exist stable minimal hypersurfaces when $n\geq4$ (see \cite{BE}) making it harder to apply similar arguments to higher dimensional cases. Moreover, their proof may not work in spaces where the scalar curvature can be negative. We discuss this matter in Section \ref{sec4}.
\end{rmk}

Normally, Frankel’s argument was applied to manifolds with positive Ricci curvature contrary to our cases (\cite{FLi,F,L}). On the other hand, Frankel’s property has been generalized to $f$-minimal hypersurfaces in smooth metric measure spaces $(M,g,e^{-f}d\text{Vol})$. In many cases, the following Bakry-\'{E}mery Ricci curvature is assumed to be positive:
\begin{align*}
\text{Ric}_f := \text{Ric} + \text{Hess} f.
\end{align*} 
For related results, see \cite{IPR} and the references therein.

A minimal hypersurface in the Schwarzschild manifold can be considered as a $f$-minimal hypersurface in the smooth metric measure space $(\mathbb{R}^n, g_{\mathbb{R}^n},e^{-f}d\text{Vol})$, where
\begin{align*}
f=-\log \left( 1+ \frac{m}{2|x|^{n-2}}\right)^{\frac{4}{n-2}}.
\end{align*}
We can calculate the Bakry-\'{E}mery Ricci curvature in this case as follows:
\begin{align*}
\text{Ric}_f (A,B)=& \frac{2m}{|x|^n \left( 1+ \frac{m}{2|x|^{n-2}}\right)}g_{\mathbb{R}^n}(A,B)\\&-\frac{2m(n|x|^{n-2}+m)}{|x|^{2n} \left( 1+\frac{m}{2|x|^{n-2}}\right)^2}g_{\mathbb{R}^n}(X,A)g_{\mathbb{R}^n}(X,B).
\end{align*}
We see that it is negative in radial directions. Furthermore as $|x| \rightarrow \infty$ we get $|\text{Ric}_f| \rightarrow 0$, so there is no positive lower bound for $|\text{Ric}_f|$. Therefore our result shows that Frankel's argument also can be applied to some non-positively curved manifolds.  

The above theorem implies the following half-space theorem for a subclass of free boundary minimal hypersurfaces in $M^n$:
\begin{cor}[Half-space theorem] \label{cor1}
Let $\Sigma^{n-1}$ and $(M^n, g^n_{Sch})$ be as in Theorem \ref{thm1}. Suppose that $\Sigma^{n-1}$ is on one side of a totally geodesic hyperplane $P$ and satisfies (\ref{dist}) with respect to $P$. Then $\Sigma^{n-1}=P$.
\end{cor}
\begin{proof}
By Theorem \ref{thm1}, we have $\Sigma^{n-1}\cap P\neq\emptyset$. Since $\Sigma^{n-1}$ is on one side of $P$, $\Sigma^{n-1}$ should touch $P$ at some points. Then the  result follows from the boundary or interior maximum principle.
\end{proof}

\begin{rmk}
\textup{It is generally difficult to classify asymptotic behaviors of complete minimal ends without additional assumptions such as conditions on total curvature, volume growth, or stability (see \cite{BR,C,CM} for more details). Therefore it can be said that (\ref{dist}) of Theorem \ref{thm1} covers, in a sense, a wide class of minimal ends.}
\end{rmk}
\section{Perturbation of the Schwarzschild metric}\label{sec4}
As we mentioned in Remark \ref{slabslab}, the slab theorem in \cite{CCE} implies that the Frankel property holds in an asymptotically flat 3-manifold with non-negative scalar curvature and horizon boundary. Since they used the stability inequality along with the Schoen-Yau trick, the sign of the scalar curvature plays an important role. 

In Theorem \ref{thm1}, we obtained the Frankel property in the Schwarzschild manifold under the restriction on the distance condition (\ref{dist}). The way we got this result was to apply Frankel's argument directly by observing the Ricci curvature along the geodesic (see Proposition \ref{ricciineq}). Since the Schwarzschild metric is scalar-flat and the integral
\begin{align*}
\int_0^{\infty}-\emph{Ric}_{M^n}(\dot{\gamma},\dot{\gamma})\d s
\end{align*}
is strictly negative, we can guess that our method can be extended to some metrics close to the Schwarzschild metric regardless of the sign of the scalar curvature. However, there still remains a question if one could remove the restriction on the distance at infinity which was possible in \cite{CCE} when $n=3$.

In the rest of this section, we provide a way to perturb the Schwarzschild metric to satisfy the same Ricci curvature inequality. Note that this perturbation contains a metric which has both positive and negative scalar curvature.

We restrict ourselves to consider rotationally symmetric metrics $g=e^{2\phi}g_{\mathbb{R}^n}$ with the following conditions:
\begin{align*}
\text{(a)}&\  e^{\phi(\frac{R^2}{r})}\cdot\frac{R^2}{r}=e^{\phi(r)}\cdot r,\ \forall r>0.\\
\text{(b)}&\  1+r\frac{\d\phi}{\d r}(r)>0,\ \forall r>R_{\phi}.
\end{align*}
The first one guarantees the existence of a horizon boundary where we denote its radius as $R_{\phi}$. By the second assumption, $u=u(r)=re^{\phi}$ has an inverse function $r=r(u)$ $(u\geq C_{\phi}:=R_{\phi}e^{\phi(R_{\phi})})$. 

Now define
\begin{align*}
f_{\phi}(u):=1+r(u)\frac{\d\phi}{\d r}(r(u)).
\end{align*}
Then $f_{\phi}(C_{\phi})=0$ by (a) and it is easy to see that $\frac{1}{uf_{\phi}(u)}$ is integrable on $u\geq C_{\phi}$.
If the geodesic $\gamma$ starts at $r=r_0$ we get
\begin{align*}
&\int_0^{\infty}-\emph{Ric}_{M^n}(\dot{\gamma},\dot{\gamma})\d s\\
=&\int_{u_0}^{\infty}\left[\left(\frac{n-1}{u}-\frac{(n-2)u_0^2}{u^3}\right)\cdot\left(u\frac{\d f_{\phi}}{\d u}\right)-\frac{(n-2)u_0^2}{u^3}\cdot B(f_{\phi})\right]\frac{1}{\sqrt{u^2-u_0^2}}\d u,
\end{align*}
where $B(x)=\frac{1}{x}-x$ and $u_0=u(r_0)$. We define $R(\phi, u_0)$ to be the right-hand side.
\begin{ex}[The Schwarzschild metric]\label{exex}
\normalfont When $\phi=\phi_{n,m}$ as in Section \ref{pre}, we have
\begin{align*}
f_{\phi}(u)=f_{n,m}(u):=\sqrt{1-\frac{2m}{u^{n-2}}}
\end{align*}
and $R_{\phi}=\left(\frac{m}{2}\right)^{\frac{1}{n-2}}$, $C_{\phi}=(2m)^{\frac{1}{n-2}}$. The computation in Section \ref{pre} implies that 
\begin{align*}
R(\phi_{n,m}, u_0)=-\sum_{j=0}^{\infty}\frac{1}{2^{2j+1}}&\cdot\binom{2j+1}{j}\cdot\frac{n-2}{2}\cdot\frac{j+1}{j+\frac{2n-2}{n-2}}\\
&\cdot\left(\int_0^{\frac{\pi}{2}}\sin^{(j+2)(n-2)}\psi\d \psi\right)\cdot\frac{2^{j+2}m^{j+2}}{u_0^{(n-2)(j+2)+1}}.
\end{align*}
\end{ex}
\begin{rmk}\label{condf}
\normalfont Once we get $f$ defined in $[C_f,\infty)$ such that $f(C_f)=0$ and $\frac{1}{uf(u)}$ is integrable on $[C_f,\infty)$, we can explicitly construct a rotationally symmetric metric $e^{2\phi}g_{\mathbb{R}^n}$ as follows. Let
\begin{align*}
h(u):=R_fe^{\int_{C_f}^u\frac{1}{xf(x)}\d x}
\end{align*}
for some $R_f>0$ and $u=u(h)$ be the inverse of $h$. Then
\begin{align*}
\phi_f(h):=\log\frac{u(h)}{h}
\end{align*} 
satisfies (a) and (b). The metric $e^{2\phi_f}g_{\mathbb{R}^n}$ has the horizon radius $R_f$ and $f_{\phi_f}=f$.
\end{rmk}
Now we get perturbations of the Schwarzschild metric with $R(\phi, u_0)<0$:
\begin{thm}
Fix $m>0$. Suppose that we have a rotationally symmetric metric $e^{2\phi}g_{\mathbb{R}^n}$ satisfying (a), (b), and $C_{\phi}\geq (2m)^{\frac{1}{n-2}}$. If 
\begin{align}\label{perturbb}
u\frac{\d f_{\phi}}{\d u}-u\frac{\d f_{n,m}}{\d u}\leq\frac{a}{u^{2n-4}}\ \text{and}\ B(f_{\phi})-B(f_{n,m})\geq-\frac{b}{u^{2n-4}}
\end{align}
for some $a,b\geq0$ satisfying
\begin{align}\label{abab}
(3n-4)a+(2n-3)(n-2)b<(n-2)^2m^2,
\end{align}
then $R(\phi, u_0)<0$ for all $u_0\geq C_{\phi}$.
\end{thm}
\begin{proof}
We first calculate the following integral as in Section \ref{pre} to obtain
\begin{align*}
\int_{u_0}^{\infty}\left[\left(\frac{n-1}{u}-\frac{(n-2)u_0^2}{u^3}\right)\cdot\frac{a}{u^{2n-4}}+\frac{(n-2)u_0^2}{u^3}\cdot\frac{b}{u^{2n-4}}\right]\frac{1}{\sqrt{u^2-u_0^2}}\d u\\
=\frac{1}{u_0^{2n-3}}\cdot\int_0^{\frac{\pi}{2}}\sin^{2n-4}\psi\d\psi\cdot\left(\frac{3n-4}{2n-2}a+\frac{(2n-3)(n-2)}{2n-2}b\right).
\end{align*}
Then the positivity of the terms
\begin{align*}
\frac{n-1}{u}-\frac{(n-2)u_0^2}{u^3},\ \frac{(n-2)u_0^2}{u^3}
\end{align*}
when $u\geq u_0$ and (\ref{perturbb}) imply that
\begin{align*}
R(\phi,u_0)\leq R(\phi_{n,m},u_0)+\frac{1}{u_0^{2n-3}}\cdot\int_0^{\frac{\pi}{2}}\sin^{2n-4}\psi\d\psi\cdot\left(\frac{3n-4}{2n-2}a+\frac{(2n-3)(n-2)}{2n-2}b\right).
\end{align*}
Note that the assumption $C_{\phi}\geq(2m)^{\frac{1}{n-2}}$ enables us to use the term $R(\phi_{n,m},u_0)$ for any $u_0\geq C_{\phi}$. By the series expansion in Example \ref{exex}, we have
\begin{align*}
R(\phi_{n,m},u_0)<-\frac{(n-2)^2m^2}{2n-2}\cdot\frac{1}{u_0^{2n-3}}\cdot\int_0^{\frac{\pi}{2}}\sin^{2n-4}\psi\d\psi,
\end{align*}
where the right-hand side corresponds to the term when $j=0$. Therefore
\begin{align*}
R(\phi,u_0)<\frac{1}{u_0^{2n-3}}\cdot\int_0^{\frac{\pi}{2}}\sin^{2n-4}\psi\d\psi\cdot\left(\frac{3n-4}{2n-2}a+\frac{(2n-3)(n-2)}{2n-2}b-\frac{(n-2)^2m^2}{2n-2}\right),
\end{align*}
and by (\ref{abab}), we can conclude that $R(\phi,u_0)<0$.
\end{proof}
If $\phi=\phi_{n,m}$ is the Schwarzschild metric, it trivially satisfies the condition with $a=b=0$. Therefore the above theorem gives a perturbation of $\phi_n$. The following example shows that there are nontrivial metrics in the above class whose scalar curvature can be both positive and negative.
\begin{ex}
\normalfont For simplicity, we only consider the case when $n=3$. Let 
\begin{align*}
E(u)=
\begin{cases}
0 & \text{if } 2m\leq u\leq3m,\\
-\frac{m}{256}(u-3m) & \text{if } 3m\leq u\leq4m,\\
\frac{m}{256}(u-5m) & \text{if } 4m\leq u\leq6m,\\
\frac{m^2}{256} & \text{if } u\geq6m.
\end{cases}
\end{align*}
After smoothing the function $E$, we denote it as $\mathcal{E}$. If we let
\begin{align*}
f(u):=f_{3,m}(u)+\frac{\mathcal{E}(u)}{u^2},\ u\geq2m,
\end{align*}
then $f$ satisfies the conditions in Remark \ref{condf} so that we obtain the metric $e^{2\phi}g_{\mathbb{R}^3}$. One can check easily that (\ref{perturbb}) and (\ref{abab}) hold for this metric with $a=b=\frac{m^2}{16}$. Moreover, since the scalar curvature is given by
\begin{align*}
\text{Scal}(u)=-\frac{2f(u)}{u^2}\left(2u\frac{\d f}{\d u}(u)-B(f(u))\right),
\end{align*}
it has negative sign when $u$ is slightly greater than $3m$ and has positive sign when $u$ is large enough.
\end{ex}


\begin{thebibliography}{0}
\bibitem{BE} Barbosa, E., Espinar, J. M.: {On free boundary minimal hypersurfaces in the Riemannian Schwarzschild space}. J. Geom. Anal., \textbf{31}, 12548--12567 (2021)
\bibitem{BR} Bernard, Y., Rivi\`ere, T.: {Ends of immersed minimal and Willmore surfaces in asymptotically flat spaces}. Comm. Anal. Geom., \textbf{28}(1), 1--57 (2020)
\bibitem{B} Brendle, S.: {Constant mean curvature surfaces in warped product manifolds}. Publ. Math. Inst. Hautes \'Etudes Sci., \textbf{117}, 247--269 (2013)
\bibitem{C} Carlotto, A.: {Rigidity of stable minimal hypersurfaces in asymptotically flat spaces}. Calc. Var. Partial. Differ. Equ., \textbf{55}(54), 53--73 (2016)
\bibitem{CCE} Carlotto, A., Chodosh, O., Eichmair, M.: {Effective versions of the positive mass theorem}. Invent. Math., \textbf{206}, 975--1016 (2016)
\bibitem{CM} Carlotto, A., Mondino, A.: {A non-existence result for minimal catenoids in asymptotically flat spaces}. J. Lond. Math. Soc., \textbf{95}(2), 373--392 (2017)
\bibitem{CK} Chodosh, O., Ketover, D.: {Asymptotically flat three-manifolds contain minimal planes}. Adv. Math., \textbf{337}, 171--192 (2018)
\bibitem{CTZ} Cris\'{o}stomo, J., Troncoso, R., Janelli, J.: {Black hole scan}. Phys. Rev. D, \textbf{62}(8), 084013 (2000)
\bibitem{F} Frankel, T.: {On the fundamental group of a compact minimal submanifold}. Ann. of Math., \textbf{83}, 68--73 (1966)
\bibitem{FLi} Fraser, A., Li, M.: {Compactness of the space of embedded minimal surfaces with free boundary in three-manifolds with nonnegative Ricci curvature and convex boundary}. J. Differ. Geom., \textbf{96}(2), 183--200 (2014)
\bibitem{HE} Hawking, S. W., Ellis, G. F. R.: {The large scale structure of space-time}. Cambridge University Press (1973)
\bibitem{HM} Hoffman, D., Meeks, W.: {The strong half-space theorem for minimal surfaces}. Invent. Math., \textbf{101}, 373--377 (1990)
\bibitem{IPR} Impera, D., Pigola, S., Rimoldi, M.: {The Frankel property for self-shrinkers from the viewpoint of elliptic PDEs}. J. Reine Angew. Math., \textbf{773}, 1--20 (2021)
\bibitem{L} Lawson, H. B.: {The unknottedness of minimal embeddings}. Invent. Math., \textbf{11}, 183--187 (1970)
\bibitem{SY} Schoen, R., Yau, S. T.: {On the proof of the positive mass conjecture in general relativity}. Comm. Math. Phys., \textbf{65}, 45--76 (1979)



\end{thebibliography}
\end{document}